\documentclass[12pt]{article}
\usepackage{graphicx}
\usepackage{amsmath}
\usepackage{amsfonts}
\usepackage{amsthm}
\usepackage[T1]{fontenc}
\usepackage{url}
\usepackage[margin=1in]{geometry}
\numberwithin{equation}{section}
\begin{document}
\title{Sharp weighted Korn and Korn-like inequalities and an application to washers}
\author{Davit Harutyunyan\\
\textit{Department of Mathematics, University of Utah}}
\maketitle

\begin{abstract}
In this paper we prove asymptotically sharp weighted "first-and-a-half" $2D$ Korn and Korn-like inequalities with a
singular weight occurring from Cartesian to cylindrical change of variables. We prove some Hardy and the so-called
"harmonic function gradient separation" inequalities with the same singular weight. Then we apply
the obtained $2D$ inequalities to prove similar inequalities for washers with thickness $h$ subject to vanishing
Dirichlet boundary conditions on the inner and outer thin faces of the washer.
A washer can be regarded in two ways: As the limit case of a conical shell when the slope goes to zero, or as a very short hollow cylinder.
While the optimal Korn constant in the first Korn inequality for a conical shell with
thickness $h$ and with a positive slope scales like $h^{1.5}$ e.g. [\ref{bib:Gra.Har.4}],
the optimal Korn constant in the first Korn inequality for a washer
scales like $h^2$ and depends only on the outer radius of the washer
as we show in the present work. The Korn constant in the first and a half inequality scales like $h$
and depends only on $h.$ The optimal Korn constant is realized by a Kirchoff Ansatz.
This results can be applied to calculate the critical buckling load of a washer under
in plane loads, e.g. [\ref{bib:Ant.Ste.}]. \\
\newline
\textbf{Keywords}\ \ Korn inequality; elasticity; thin domains; shells; plates; washers
\newline
\linebreak
\textbf{Mathematics Subject Classification}\ \ 00A69, 74B05, 74B20, 74K20, 74K25
\end{abstract}

\section{Introduction}
\label{sec:1}

\newtheorem{Theorem}{Theorem}[section]
\newtheorem{Lemma}[Theorem]{Lemma}
\newtheorem{Corollary}[Theorem]{Corollary}
\newtheorem{Remark}[Theorem]{Remark}
\newtheorem{Definition}{Definition}[section]
Korn's inequalities have arisen in the investigation of the boundary value problem
of linear elastostatics, [\ref{bib:Korn.1},\ref{bib:Korn.2}] and have been proven by
different authors, e.g. [\ref{bib:Kon.Ole.1},\ref{bib:Kon.Ole.2},\ref{bib:Friedrichs},\ref{bib:Pay.Wei.},\ref{bib:Kohn}].
Some generalized versions of the classical second Korn inequality have been recently proven in [\ref{bib:Nef.Pau.Wit.1},\ref{bib:Nef.Pau.Wit.2},\ref{bib:Con.Dol.Mul.},\ref{bib:Bau.Nef.Pau.Sta.}].
Traditionally there are two types of Korn inequalities spoken about, the first and the second one.
The classical first Korn inequality is typically formulated as follows:
\textit{Assume $\Omega\in\mathbb R^n$ is a simply connected open domain with a Lipschitz boundary,
and $V\subset W^{1,2}(\Omega)$ is a closed subspace that contains no rigid motion other than the identically zero one. Then,
there exists a constant $C_1(\Omega, V)$ depending only on $\Omega$ and $V$ such that the inequality
$$
C_1\|\nabla \mathbf{U}\|^2\leq \|e(\mathbf{U})\|^2
$$
holds for all displacement fields $\mathbf{U}\in V$.} Accordingly, the quantity
\begin{equation}
\label{1.05}
K(\Omega,V)=\inf_{\mathbf{U}\in V}\frac{\|e(\mathbf{U})\|^2}{\|\nabla \mathbf{U}\|^2}
\end{equation}
will be called the Korn's constant of the domain $\Omega$ associated to the subspace $V.$
The classical second Korn inequality reads as follows: \textit{Assume $\Omega\in\mathbb R^n$ is a simply connected open domain with a
Lipschitz boundary, then there exists a constant $C_2(\Omega)$ depending only on $\Omega$ such that the inequality
\begin{equation}
C_2\|\nabla \mathbf{U}\|^2\leq \|e(\mathbf{U})\|^2+\|\mathbf{U}\|^2
\end{equation}
holds for all displacement fields $\mathbf{U}\in\mathbb W^{1,2}(\Omega, \mathbb R^n),$ where
$e(\mathbf{U})=\frac{1}{2}(\nabla \mathbf{U}+\nabla \mathbf{U}^T)$
is the symmetrized gradient, i.e., strain in the linear elasticity context.}
We refer to the review article [\ref{bib:Horgan}] for a detailed discussion of Korn inequalities
and their role in the theory of linear and non-linear elasticity.
A more global variant of the first Korn inequality has been proven by Kohn in [\ref{bib:Kohn}].
Korn inequalities and other related inequalities for integrals of quadratic functionals also
arise in the analysis of viscous incompressible fluid flow, see the references in [\ref{bib:Horgan}], while Korn-like inequalities with
tangential boundary conditions arise in statistical mechanics, [\ref{bib:Des.Vil.},\ref{bib:Lew.Mul.1}].
In many applications it is essential to know the dependence of the optimal Korn
constants $C_1$ and $C_2$ upon the geometric parameters of the application domain $\Omega$ and the subspace $V.$
We will call such kind of Korn inequalities with optimal constants\footnote{In general it is not known whether a best constant in a
first or second Korn inequality exists. Here we speak about asymptotic optimality.}
sharp Korn inequalities. Sharp Korn inequalities have been derived by several authors e.g.,
[\ref{bib:Dau.Sur.},\ref{bib:Koh.Vog.},\ref{bib:Gra.Tru.},\ref{bib:Lew.Mul.1},\ref{bib:Lew.Mul.2},\ref{bib:Par.Tom.1},\ref{bib:Par.Tom.2},\ref{bib:Ryzhak},\ref{bib:Naz.Slu.},\ref{bib:Naz.1},\ref{bib:Naz.2},\ref{bib:Naz.3},\ref{bib:Par.Tom.1},\ref{bib:Par.Tom.2},\ref{bib:Harutyunyan},\ref{bib:Gra.Har.1},\ref{bib:Gra.Har.2},\ref{bib:Gra.Har.3}].
One of the applications of such inequalities is the study of buckling of slender structures
[\ref{bib:Dau.Sur.},\ref{bib:Horgan},\ref{bib:Gra.Tru.},\ref{bib:Par.Tom.2},\ref{bib:Gra.Har.2},\ref{bib:Gra.Har.3}],
where Korn's constants govern the scaling of the critical load as a function of the slenderness parameter as understand by
Grabovsky and Truskinovsky in [\ref{bib:Gra.Tru.}]. In [\ref{bib:Gra.Tru.}] the theory of buckling of slender structures of
Grabovsky and Truskinovsky deals with the buckling of a slender body $\Omega_h$ (parameterized by the small parameter
$h,$ usually the thickness of the body $\Omega_h$) under dead loads $\mathbf{t}(x,h,\lambda)$ with magnitude $\lambda$ applied to the boundary of $\Omega_h.$
The buckling is then roughly speaking defined as the first loss of stability
of the trivial branch (resulting Lipshitz deformation) $\mathbf{y}(x,h,\lambda),$ i.e., when the
second variation of the total energy
$$E(\mathbf{y})=\int_{\Omega_h}W(\nabla \mathbf{y}(x))dx-\int_{\partial\Omega_h}\mathbf{y}(s)\cdot \mathbf{t}(s)ds$$
becomes negative at some variations $\phi.$ The admissible set of variations $\phi$ is determined by the dead loads $\mathbf{t}(x,h,\lambda),$
namely the load $\mathbf{t}(x,h,\lambda)$ gives rise to Dirichlet type vanishing boundary conditions for the variations $\mathbf{\phi}$ on the boundary $\partial\Omega_h.$
Thus $\mathbf{\phi}\in V_h=V(\Omega_h),$ where $V_h\in W^{1,\infty}(\Omega_h),$ for some subspace $V_h.$
Grabovsky and Harutyunyan give an alternative explicit formula for the critical buckling load
of a slender body under dead loads in [\ref{bib:Gra.Har.3}], utilizing which they study the buckling of
cylindrical shells under axial compression in [\ref{bib:Gra.Har.2}] and under axial compression with
some additional torque\footnote{Note, that in general the load need not be hydrostatic.} in [\ref{bib:Gra.Har.3}].
For the sake of clearness, we recall the formula mentioned above. Denote
\begin{equation}
\label{1.1}
R(h,\mathbf{\phi})=\frac{\int_{\Omega_{h}}(L_{0}e(\mathbf{\phi}),e(\mathbf{\phi}))dx}
{\int_{\Omega_{h}}(\sigma_{h},\nabla\mathbf{\phi}^{T}\nabla\mathbf{\phi})dx},
\end{equation}
where $L_0=W_{FF}(I)$ is the linear elasticity matrix, $\sigma_h$ is the Piola-Kirchoff stress tensor and
$e(\phi)=\frac{1}{2}\left(\nabla\mathbf{\phi}+\nabla\mathbf{\phi}^T\right)$ is the linear elastic strain.
Define furthermore
\begin{equation}
\label{1.2}
\lambda(h)=\inf_{\mathbf{\phi}\in V_h}R(h,\mathbf{\phi}).
\end{equation}
Then the following theorem holds:
\begin{Theorem}[The critical load]
Assume that the quantity $\lambda(h)$ defined in (\ref{1.2})  satisfies
$\lambda(h)>0$ for all sufficiently small $h$ and
\begin{equation}
 \label{1.3}
\lim_{h\to 0}\frac{{\lambda}(h)^{2}}{K(\Omega_h,V_{h})}=0,
\end{equation}
where $K(\Omega_h,V_h)$ is the Korn's constant of the domain $\Omega_h$ associated to the subspace $V_h.$
Then $\lambda(h)$ is the critical buckling load and the variational problem (\ref{1.2})
captures the buckling modes (deformations) too (see [\ref{bib:Gra.Har.3}] for a precise definition of buckling modes).
\end{Theorem}
Note, that $\lambda(h)$ is the infimum of a quadratic form in $e(\mathbf{\phi})$ over a quadratic form in $\nabla\mathbf{\phi},$ and thus
is closely related to the Korn's constant $K(\Omega_h,V_h).$ It will have the same asymptotics as $K(\Omega_h,V_h)$ if
for instance the stress tensor $\sigma_h$ is a full rank $3\times3$ matrix, but could decay slower than
$K(\Omega_h,V_h)$ otherwise as $h\to 0$ as observed in [\ref{bib:Gra.Har.2}]. Therefore, in order to be able to calculate
$\lambda(h)$ one should be able to calculate the Korn's constant $K(\Omega_h,V_h).$
With this applications in mind we study the scaling of the Korn's constant as a power of the thickness $h$ of a washer
under Dirichlet type boundary conditions.
In [\ref{bib:Dafermos}], the Dafermos proved second Korn inequalities for rings under some kind of normalization condition,
which is the average of some fiend over the application domain vanishes, i.e, a global condition.
In this work we prove weighted first and first and a half Korn and Korn-type inequalities with optimal constants
for washers. We First prove two dimensional weighted first and a half Korn and Korn-type
inequalities with exact constants subject to Dirichlet boundary conditions and then apply them
on the cross sections of the washer to derive the targeted three dimensional inequalities.
The new weighted inequalities seem to be the appropriate tool for treating
inequalities in cylindrical coordinates due to the presence of the Jacobian $\rho.$
It is shown in the present work for washers with thickness $h,$
that the optimal Korn's constant in the first and a half Korn inequality scales like $h$ and does not depend
on the either of inner and outer radii of the washer. In the first Korn
inequality, the optimal constant scales like $h^2$ and depends only on $h$ and the outer radius $R$ of the washer.
The present results can be used to calculate the critical buckling load of washers under
in-plane pressure applied to the thin part of the boundary as done for cylindrical shells under axial compression in
[\ref{bib:Gra.Har.2},\ref{bib:Gra.Har.3}]. See also a similar work of Antman and Stepanov [\ref{bib:Ant.Ste.}].

\section{Main Results}
\label{sec:2}
In this section we formulate the main results of the paper, while the proofs will be presented in Section~\ref{sec:4} and Section~\ref{sec:5}.
Denote the displacement vector by $\mathbf{u}=(u_\rho, u_\theta, u_z)$ in cylindrical coordinates.
Then the Cartesian gradient of $\mathbf{u}$ has the form
\begin{equation}
\label{2.1}
\nabla \mathbf{u}=
\begin{bmatrix}
u_{\rho,\rho} & \frac{u_{\rho,\theta}-u_\theta}{\rho} & u_{\rho,z}\\
u_{\theta,\rho} & \frac{u_{\theta,\theta}+u_\rho}{\rho} & u_{\theta,z}\\
u_{z,\rho} & \frac{u_{z,\theta}}{\rho} & u_{z,z},\\
\end{bmatrix},
\end{equation}
where $f_{,x}$ means the partial derivative of $f$ with respect to the variable $x.$ We will use different
notations for partial derivatives in the sequel interchangeably.
It is clear, that a washer with thickness $h,$ and inner and outer radii $r$ and $R$ is given
in cylindrical coordinates by $\Omega=\{(\rho,\theta,z) \ : \ \rho\in[r,R], \theta\in[0,2\pi], z\in[0,h]\}.$
Consider the subspaces
\begin{align}
\label{2.1.5}
V_1=\{\mathbf{u}\in W^{1,2}(\Omega) \ : \ u_\theta(r,\theta,z)=u_\theta(R,\theta,z)=u_\rho(r,\theta,z)=u_\rho(R,\theta,z)=0\},\\ \nonumber
V_2=\{\mathbf{u}\in W^{1,2}(\Omega) \ : \ u_\theta(r,\theta,z)=u_\theta(R,\theta,z)=u_z(r,\theta,z)=u_z(R,\theta,z)=0\},
\end{align}
i.e., we impose zero boundary conditions on the angular and vertical or radial components
of the displacement $\mathbf{u}$ (depending on the character of the load the boundary conditions may change)
on the inner and outer thin faces of the boundary of the washer.
As the Jacobin of a cylindrical change of variables is $\rho,$ then the Cartesian norm of a vector $\mathbf{F}$ will be
$\|\sqrt{\rho}\mathbf{F}\|_{L^2(\Omega)}$ in cylindrical coordinates. In what follows we will work in cylindrical coordinates
for all $3D$ Korn inequalities, and thus will deal with norms of the form $\|\sqrt{\rho}\mathbf{F}\|_{L^2(\Omega)}.$ We will
also sometimes leave out the $L^2(\Omega)$ in the norm notation when it is clear which domain is under consideration.
The following weighted inequality on the gradient of a harmonic function in rectangles is crucial in proving
appropriate weighted Korn inequalities in two dimensional rectangles. We call this kind of inequalities
\textit{harmonic function gradient separation estimates.}
\begin{Theorem}[Harmonic function gradient separation inequality]
\label{th:2.0}
Assume $L>l>0,$ $h>0$ and $T=(0,h)\times(l,L).$ Assume furthermore, that $f(x,y)\in C^2(\overline{T})$ is harmonic in $T$
and satisfies the boundary conditions $f(x,l)=f(x,L)=0$ for all $x\in[0,h].$
Then there holds
\begin{equation}
\label{2.1.6}
\|\sqrt{y}f_{,y}\|_{L^2(T)}^2\leq 40\left(\frac{\|\sqrt{y}f_{,x}\|_{L^2(T)}\cdot\|\sqrt{y}f\|_{L^2(T)}}{h}+\|\sqrt{y}f_{,x}\|_{L^2(T)}^2\right).
\end{equation}
\end{Theorem}

The next one is the appropriate weighted first and a half Korn inequality on rectangles.
\begin{Theorem}[Weighted first and a half Korn inequality for rectangles]
\label{th:2.0.5}
Let $L>l>0$ and $h>0$ satisfy $h\leq cl$ for some $c>0$ and denote $T=(0,h)\times(l,L).$
Then there exists a constant $C=C(c,L)>0$ depending only on $c$ and $L$ such, that for any displacement
$\mathbf{U}=(f,g)\in W^{1,2}(\overline{T},\mathbb R^2)$ satisfying the boundary conditions $f(x,l)=f(x,L)=0,$ for all $x\in[0,h]$
in the sense of traces, the weighted first and a half Korn inequality holds:
\begin{equation}
\label{3.16}
\|\sqrt{y}\nabla \mathbf{U}\|_{L^2(T)}^2\leq
C\left(\frac{\|\sqrt{y}f\|_{L^2(T)}^2\cdot\|\sqrt{y}e(\mathbf{U})\|_{L^2(T)}^2}{h}+\|\sqrt{y}e(\mathbf{U})\|_{L^2(T)}^2\right).
\end{equation}
\end{Theorem}

\begin{Theorem}[First and a half Korn inequality]
\label{th:2.1}
Assume $h\leq cr$ for some $c>0.$ Then there exists a constant $C=C(c,R)>0$ depending only on $c$ and $R$ such that
for any displacement $\mathbf{u}\in V_1\cup V_2$ the inequality holds:
\begin{equation}
\label{2.2}
\|\sqrt{\rho}\nabla \mathbf{u}\|^2\leq C\left(\frac{\|\sqrt{\rho}u_z\|\cdot\|\sqrt{\rho}e(\mathbf{u})\|}{h}+\|\sqrt{\rho}e(\mathbf{u})\|^2\right).
\end{equation}
\end{Theorem}
Let us remark, that the term "first and a half" Korn inequality was introduced in [\ref{bib:Gra.Har.1}], where the authors prove such
inequalities for $2D$ rectangles and then cylindrical shells. It is also interesting, that the first and a half Korn
inequality implies both, the first and the second Korn inequalities due to the Friedrichs inequality.
\begin{Theorem}[First Korn inequality]
\label{th:2.2}
Assume the requirements of Theorem~\ref{th:2.1} are fulfilled. Then
there exists a constant $C=C(c,R)>0$ depending only on $c$ and $R$
such that for any displacement $\mathbf{u}\in V_1\cup V_2$ the inequality holds:
\begin{equation}
\label{2.3}
\|\sqrt{\rho}\nabla \mathbf{u}\|^2\leq\frac{C}{h^2}\|\sqrt{\rho}e(\mathbf{u})\|^2.
\end{equation}
\end{Theorem}

\begin{Theorem}[Realizability of the asymptotics]
\label{th:2.3}
Both inequalities (\ref{2.2}) and (\ref{2.3}) are sharp in the sense that the exponents $1$ and $2$ of $h$ in them are optimal as $h$ goes to zero.
\end{Theorem}

\section{Hardy and Korn-like inequalities}
\label{sec:3}
In this section we recall two results proven in [\ref{bib:Harutyunyan},\ref{bib:Gra.Har.1}] and prove new
suitable Hardy and weighted Hardy-like inequalities.
The next theorem is proven in [\ref{bib:Gra.Har.1}, Theorem~3.1] for rectangles and generalized in [\ref{bib:Harutyunyan}]
for any $2D$ curved thin domains with a varying thickness. For a proof we refer to [\ref{bib:Gra.Har.1},\ref{bib:Harutyunyan}].
\begin{Theorem}
\label{th:3.1}
Assume $L>0,$ $h\in(0,1)$ and denote $T=(0,h)\times(0,L).$ Assume the displacement
$\mathbf{U}=(f,g)\in W^{1,2}(T)$ satisfies one of following the boundary condition:
\begin{itemize}
\item[(i)] $g(x,0)=0,$ for all $x\in(0,h)$ in the sense of traces,
\item[(ii)] $f(x,0)=f(x,L),$ for all $x\in(0,h)$ in the sense of traces.
\end{itemize}
Then the first and a half Korn inequality holds:
\begin{equation}
\label{3.1}
\|\nabla \mathbf{U}\|^2\leq 100\left(\frac{\|f\|\cdot\|e(\mathbf{U})\|}{h}+\|e(\mathbf{U})\|^2\right).
\end{equation}
\end{Theorem}

The next lemma is the weighted version of Lemma~2.1 in [\ref{bib:Harutyunyan}], see also Lemma~2.2 in [\ref{bib:Ole.Sha.Yos.}].
\begin{Lemma}
\label{lem:3.3}
For $L>l>0$ and $h>0$ denote $T=(0,h)\times(l,L).$ Assume that the function $f\in C^2(\overline{T},\mathbb R)$ is harmonic in $T$ and
satisfies the boundary conditions $f(x,l)=f(x,L)=0$ for all $x\in[0,h].$ Denote $\delta(x,y)=\min(x,h-x)$ for $(x,y)\in \overline{T}.$ Then
there holds
\begin{equation}
\label{3.5}
\|\sqrt{y}\delta\nabla f\|_{L^2(T)}^2\leq 4\|\sqrt{y}f\|_{L^2(T)}^2.
\end{equation}
\end{Lemma}

\begin{proof}
Integrating by parts and utilizing the boundary conditions on $f$ and $\delta,$ the harmonicity of $f$ and the facts that
$\delta$ is independent of $y$ and $|\frac{\partial\delta}{\partial x}|=1,$ we obtain

\begin{align}
\label{3.6}
\|\sqrt{y}\delta\nabla f\|_{L^2(T)}^2&=\int_{T}y\delta^2|\nabla f|^2dxdy\\ \nonumber
&=-\int_{T}f\frac{\partial}{\partial x}\left(y\delta^2\frac{\partial}{\partial x}\right)
-\int_{T}f\frac{\partial}{\partial y}\left(y\delta^2\frac{\partial}{\partial y}\right)\\ \nonumber
&=-\int_{T}fy\delta^2\triangle f-2\int_{T}fy\delta\frac{\partial \delta}{\partial x}\frac{\partial f}{\partial x}-
\int_{T}f\delta^2\frac{\partial f}{\partial y}\\ \nonumber
&=-2\int_{T}fy\delta\frac{\partial \delta}{\partial x}\frac{\partial f}{\partial x}.
\end{align}
By the arithmetic and geometric mean inequality we have
\begin{align}
\label{3.7}
\left|2\int_{T}fy\delta\frac{\partial \delta}{\partial x}\frac{\partial f}{\partial x}\right|&\leq
2\int_{T}yf^2+\frac{1}{2}\int_{T}y\left|\delta\frac{\partial f}{\partial x}\right|^2\\ \nonumber
&\leq 2\int_{T}yf^2+\frac{1}{2}\int_{T}y|\delta\nabla f|^2,
\end{align}
thus combining (\ref{3.6}) and (\ref{3.7}) we arrive at (\ref{3.5}).
\end{proof}

The next lemma is a Hardy inequality proven in [\ref{bib:Harutyunyan}, Lemma~2.4], that was inspired by the work of
Kondratiev and Oleinik in [\ref{bib:Kon.Ole.1},\ref{bib:Kon.Ole.2}]. The proof is very similar to the proof of Lemma~\ref{lem:3.2},
see [\ref{bib:Harutyunyan}] for details.
\begin{Lemma}
\label{lem:3.4}
Let $b>a>0,$ $\epsilon\in(0,1],$ and let $f\colon[a,b]\to\mathbb R$ be absolutely continuous. Then
\begin{equation}
\label{3.8}
\int_{a+\epsilon(b-a)}^bf^2(t)d t\leq \frac{2}{\epsilon}\int_a^{a+\epsilon(b-a)}f^2(t)d t+4\int_a^{b}f'^2(t)(b-t)^2d t.
\end{equation}
\end{Lemma}
The last lemma is again a Hardy inequality that will be utilized in the proof of Theorem~\ref{th:2.2}.

\begin{Lemma}
\label{lem:3.5}
Let $R>2r>0$ and let the function $f\colon[r,R]\to\mathbb R$ be absolutely continous in $[r,R]$ such that $f(r)=0.$ Then there holds
\begin{equation}
\label{3.8.1}
\int_r^{\frac{R+r}{2}}tf^2(t)dt\leq 4\int_{\frac{R+r}{2}}^R tf^2(t)dt+R^2\int_r^R tf'^2(t)dt.
\end{equation}
\end{Lemma}

\begin{proof}
For any $x\in[\frac{R+r}{2},R]$ we have by the integration by parts formula, that
$$
\int_r^x tf^2(t)dt=x^2f^2(x)-\int_r^x [tf^2(t)+2t^2f(t)f'(t)]dt,
$$
thus we get by the geometric and arithmetic mean inequality, that
\begin{align*}
\int_r^x tf^2(t)dt&=\frac{x^2f^2(x)}{2}-\int_r^x t^2f(t)f'(t)dt\\
&\leq\frac{x^2f^2(x)}{2}+\frac{1}{2}\int_r^x [tf^2(t)+t^3f'^2(t)]dt\\
&\leq\frac{x^2f^2(x)}{2}+\frac{1}{2}\int_r^x tf^2(t)dt+\frac{R^2}{2}\int_r^xtf'^2(t)dt,
\end{align*}
from where we get
\begin{equation}
\label{3.8.2}
\int_r^x tf^2(t)dt\leq x^2f^2(x)+R^2\int_r^xtf'^2(t)dt.
\end{equation}
By the mean value theorem the point $x\in[\frac{R+r}{2},R]$ can be chosen such that
$$x^2f^2(x)=\frac{2}{R-r}\int_{\frac{R+r}{2}}^R t^2f^2(t)dt,$$
thus due to the inequality $R>2r$ we get
\begin{equation}
\label{3.8.3}
x^2f^2(x)\leq \frac{2R}{R-r}\int_{\frac{R+r}{2}}^R tf^2(t)dt\leq 4\int_{\frac{R+r}{2}}^R tf^2(t)dt.
\end{equation}
A combination of the estimates (\ref{3.8.2}) and (\ref{3.8.3}) completes the proof.
\end{proof}
\section{Inequalities in two dimensions}
\label{sec:4}
\begin{proof}[Proof of Theorem~\ref{th:2.0}]
Note that this is the weighted version of Lemma~4.1 in [\ref{bib:Gra.Har.1}] but can not be derived from
it as the term $\sqrt{y}$ is not controllable from below as $l,h\to 0.$ Denote $T_t=(\frac{h}{2}-t,\frac{h}{2}+t)\times(l,L)$ and
$T_t'=(0,t)\times(l,L)$ for all $t\in(0,\frac{h}{2}).$ We have integrating by parts, that

\begin{align}
\label{3.10}
&\int_{T_t}y|\nabla f|^2dxdy=\int_{T_t}y\left(\frac{\partial f}{\partial x}\cdot\frac{\partial f}{\partial x}
+\frac{\partial f}{\partial y}\cdot\frac{\partial f}{\partial y}\right)dxdy\\ \nonumber
&=\int_{\{\frac{h}{2}+t\}\times[l,L]}yf\frac{\partial f}{\partial x}dy-\int_{\{\frac{h}{2}-t\}\times[l,L]}yf\frac{\partial f}{\partial x}dy-
\int_{T_t}y\frac{\partial^2 f}{\partial x^2}dxdy-\int_{T_t}f\left(y\frac{\partial^2 f}{\partial y^2}+\frac{\partial f}{\partial y}\right)dxdy\\ \nonumber
&=\int_{\{\frac{h}{2}+t\}\times[l,L]}yf\frac{\partial f}{\partial x}dy-\int_{\{\frac{h}{2}-t\}\times[l,L]}yf\frac{\partial f}{\partial x}dy,
\end{align}
thus we get
\begin{equation}
\label{3.11}
\int_{T_t}y|\nabla f|^2dxdy\leq
\int_{\{\frac{h}{2}+t\}\times[l,L]}\left|yf\frac{\partial f}{\partial x}\right|dy+
\int_{\{\frac{h}{2}-t\}\times[l,L]}\left|yf\frac{\partial f}{\partial x}\right|dy.
\end{equation}
Integrating the last inequality in $t$ from $0$ to $\frac{h}{2}$ and appying the Schwartz inequality we obtain
$$
\int_0^\frac{h}{2}\int_{T_t}y|\nabla f|^2dxdydt\leq \|\sqrt{y}f_{,x}\|_{L^2(T)}\cdot\|\sqrt{y}f\|_{L^2(T)}.
$$
As the integral $g(t)=\int_{T_t}y|\nabla f|^2dxdy$ is increasing in $t\in[0,\frac{h}{2}],$ then we get from the last inequality, that
\begin{equation}
\label{3.12}
\int_{T_{\frac{h}{4}}}y|\nabla f|^2\leq \frac{4}{h}\|\sqrt{y}f_{,x}\|_{L^2(T)}\cdot\|\sqrt{y}f\|_{L^2(T)}.
\end{equation}
Next we fix any point $y\in[l,L]$ and apply Lemma~\ref{lem:3.4} to the function $\sqrt{y}\frac{\partial f}{\partial y}(x,y)$
on the segment with the endpoints at $(0,y)$ and $(\frac{h}{2},y)$ as a function in $x$ and with $\epsilon=\frac{1}{2}.$ Therefore we get
\begin{equation*}
\int_{0}^{\frac{h}{4}} y|f_{,y}(x,y)|^2dx\leq 4\int_{\frac{h}{4}}^{\frac{h}{2}} y|f_{,y}(x,y)|^2dx+4\int_0^{\frac{h}{2}} x^2y|f_{,xy}(x,y)|^2dx,
\end{equation*}
integrating which in $y\in[l,L]$ we get
\begin{equation}
\label{3.13}
\int_{T_{\frac{h}{4}}'} y|f_{,y}(x,y)|^2\leq 4\int_{T_{\frac{h}{2}}'\setminus T_{\frac{h}{4}}'} y|f_{,y}(x,y)|^2+
4\int_{T_{\frac{h}{2}}'}yx^2|f_{,xy}(x,y)|^2.
\end{equation}
Next we apply Lemma~\ref{lem:3.3} to the function $f_{,x}$ in the domain $T.$ It is clear that $\delta=x$ when $(x,y)\in T_{\frac{h}{2}}'$ and that
$\triangle(f_{,x})=0$ in $T,$ thus we get
\begin{equation}
\label{3.14}
\int_{T_{\frac{h}{2}}'}yx^2|f_{,xy}(x,y)|^2\leq \int_{T}y\delta ^2|\nabla f_{,x}(x,y)|^2\leq 4\int_{T}y|f_{,x}(x,y)|^2.
\end{equation}
Combining now the estimates (\ref{3.13}) and (\ref{3.14}) we get
\begin{equation*}
\int_{T_{\frac{h}{4}}'} y|f_{,y}(x,y)|^2\leq 4\int_{T_{\frac{h}{2}}'\setminus T_{\frac{h}{4}}'} y|f_{,y}(x,y)|^2
+16\int_{T}y|f_{,x}(x,y)|^2,
\end{equation*}
summing which with (\ref{3.12}) and applying again (3.12), we finally disciver
\begin{align}
\label{3.15}
\int_{T_{\frac{h}{2}}'} y|f_{,y}(x,y)|^2&\leq 4\int_{T_{\frac{h}{2}}'\setminus T_{\frac{h}{4}}'} y|f_{,y}(x,y)|^2
+16\int_{T}y|f_{,x}(x,y)|^2+\frac{4}{h}\|\sqrt{y}f_{,x}\|_{L^2(T)}\cdot\|\sqrt{y}f\|_{L^2(T)}\\ \nonumber
&\leq 4\int_{T_{\frac{h}{4}}} y|f_{,y}(x,y)|^2+16\int_{T}y|f_{,x}(x,y)|^2+\frac{4}{h}\|\sqrt{y}f_{,x}\|_{L^2(T)}\cdot\|\sqrt{y}f\|_{L^2(T)}\\ \nonumber
&\leq 16\int_{T}y|f_{,x}(x,y)|^2+\frac{20}{h}\|\sqrt{y}f_{,x}\|_{L^2(T)}\cdot\|\sqrt{y}f\|_{L^2(T)}.
\end{align}
A similar inequality for the right half of the rectangle $T$ is analogous, thus the proof is finished.
\end{proof}

\begin{proof}[Propf of Theorem~\ref{th:2.0.5}]
It is a well established technique, that when proving a Korn inequality in a domain $\Omega$, one passes from the first $f$ component of the displacement
to its harmonic part $s$ and proves the inequality for $s,$ which then yields the inequality for $f$
as the norm $\|f-s\|_{W_{1,2}}$ is controlled by the norm $\|e(\mathbf{U})\|_{L^2}.$ It turns out, that this technique applies also for
the above weighted first and a half Korn inequality in a thin rectangle $T=(0,h)\times(l,L)$ with the weight $y,$ which
is the appropriate inequality for a washer angular cross section as will be seen in Section~\ref{sec:4}.
We are going to use that technique combined with Theorem~\ref{th:2.0} in the proof. First of all by a standard density argument
we can without loss of generality assume that $U$ is of class $C^2$ up to the boundary of $T.$ Consider then the harmonic part of
$f,$ i.e., the solution of the Dirichlet initial value problem:
\begin{equation}
\label{3.17}
\begin{cases}
\triangle s(x,y)=0, & (x,y)\in T, \\
s(x,y)=f(x,y), & (x,y)\in \partial T.
\end{cases}
\end{equation}
Then it is clear, that
$$\triangle(f-s)=\triangle f=(f_{,x})_{,x}+(f_{,y}+g_{,x})_{,y}-(g_{,y})_{x}=(e_{11}(\mathbf{U})-e_{22}(\mathbf{U}))_{,x}+2(e_{12}(\mathbf{U}))_{,y},$$
thus we have by an integration by parts,
\begin{align*}
\int_{T}&y|\nabla(f-s)|^2dxdy=\int_{T}y((f-s)_{,xx}+(f-s)_{,yy})dxdy\\
&=-\int_{T}(f-s)(y(f-s)_{,x})_{,x}+(f-s)(y(f-s)_{,y})_{,y}dxdy\\
&=-\int_{T}(f-s)((y(f-s)_{,x})_{,x}+(y(f-s)_{,y})_{,y})dxdy\\
&=-\int_{T}[(f-s)y\triangle(f-s)+(f-s)(f-s)_{,y}]dxdy\\
&=-\int_{T}y(f-s)[(e_{11}(\mathbf{U})-e_{22}(\mathbf{U}))_{,x}+2(e_{12}(\mathbf{U}))_{,y}]dxdy\\
&=\int_{T}y(f-s)_{,x}(e_{11}(\mathbf{U})-e_{22}(\mathbf{U}))+2\int_{T}y(f-s)_{,y}e_{12}(\mathbf{U})+2\int_{T}(f-s)e_{12}(\mathbf{U})dxdy\\
\end{align*}
hence we get by the Schwartz inequality,
\begin{align}
\label{3.18}
\|\sqrt{y}\nabla(f-s)\|^2&\leq 3\|\sqrt{y}\nabla(f-s)\|\|\sqrt{y}e(\mathbf{U})\|+2\left\|\frac{f-s}{\sqrt{y}}\right\|\|\sqrt{y}e(\mathbf{U})\|\\ \nonumber
&\leq 3\|\sqrt{y}\nabla(f-s)\|\|\sqrt{y}e(\mathbf{U})\|+\frac{2}{\sqrt l}\|f-s\|\|\sqrt{y}e(\mathbf{U})\|\\ \nonumber
&\leq 3\|\sqrt{y}\nabla(f-s)\|\|\sqrt{y}e(\mathbf{U})\|+\frac{2}{\sqrt{ch}}\|f-s\|\|\sqrt{y}e(\mathbf{U})\|.
\end{align}
On the other hand as $f-s$ vanishes on the entire $\partial T,$ we have for a fixed $y\in(l,L),$ that
\begin{equation*}
\sqrt{y}(f-s)(x,y)=\int_0^x\sqrt{y}(f-s)_{,t}(t,y)dt,
\end{equation*}
thus we get by the Schwartz inequality, that
\begin{equation*}
y(f-s)^2(x,y)\leq x \int_0^x y(f-s)_{,t}^2(t,y)dt\leq h\int_0^h y(f-s)_{,t}^2(t,y)dt,
\end{equation*}
which integrating first in $x\in(0,h)$ then in $y\in(l,L)$ we arrive at
\begin{equation}
\label{3.19}
\|\sqrt{y}(f-s)\|^2\leq h^2 \|\sqrt{y}\nabla(f-s)\|^2.
\end{equation}
Similarly we have
\begin{equation}
\label{3.19.5}
\|(f-s)\|\leq C\sqrt h\|\sqrt{y}\nabla(f-s)\|,
\end{equation}
for some constant $C$ depending only on $c.$
Combining now (\ref{3.18}) and (\ref{3.19}) and \ref{3.19.5} we obtain
\begin{align}
\label{3.20}
\|\sqrt{y}\nabla(f-s)\|&\leq C\|\sqrt{y}e(\mathbf{U})\|\\ \nonumber
\|\sqrt{y}(f-s)\|&\leq Ch\|\sqrt{y}e(\mathbf{U})\|.
\end{align}
Next we apply Theorem~\ref{th:2.0} to the harmonic function $s$ to get
\begin{equation*}
\|\sqrt{y}s_{,y}\|^2\leq 40\left(\frac{\|\sqrt{y}s_{,x}\|\cdot\|\sqrt{y}s\|}{h}+\|\sqrt{y}s_{,x}\|^2\right),
\end{equation*}
from where and the estimates in (\ref{3.20}) we get for sufficiently small $h,$ that
\begin{align}
\label{3.21}
&\|\sqrt{y}f_{,y}\|^2\leq 2 \|\sqrt{y}s_{,y}\|^2+2\|\sqrt{y}\nabla(f-s)\|^2\\ \nonumber
&\leq 40\left(\frac{\|\sqrt{y}s_{,x}\|\cdot\|\sqrt{y}s\|}{h}+\|\sqrt{y}s_{,x}\|^2\right)+8h^2\|\sqrt{y}e(\mathbf{U})\|\\ \nonumber
&\leq 40\left(\frac{(\|\sqrt{y}(s_{,x}-f_{,x})\|+\|f_{,x}\|)(\|\sqrt{y}(s-f)\|+\|\sqrt{y}f\|)}{h}
+2\|\sqrt{y}(s-f)_{,x}\|^2+2\|\sqrt{y}f_{,x}\|^2\right)\\ \nonumber
&+8h^2\|\sqrt{y}e(\mathbf{U})\|\\ \nonumber
&\leq C\left(\frac{\|\sqrt{y}f_{,x}\|\cdot\|\sqrt{y}f\|}{h}+\|\sqrt{y}e(\mathbf{U})\|^2\right)\\ \nonumber
&\leq C\left(\frac{\|\sqrt{y}f\|\cdot\|\sqrt{y}e(\mathbf{U})\|}{h}+\|\sqrt{y}e(\mathbf{U})\|^2\right).
\end{align}
The other secondary diagonal component $g_{,x}$ of the symmetrized gradient $e(\mathbf{U})$ is estimated in terms of $f_{,y}$
and the term $e_{12}(\mathbf{U})$ of the symmetrized gradient $e(\mathbf{U})$ by the triangle inequality, thus inequality (\ref{3.21})
finishes the proof.
\end{proof}

\section{Proofs of Theorems~\ref{th:2.1}-\ref{th:2.3}.}
\label{sec:5}
We start by recalling the main strategy for proving $3D$ Korn inequalities developed in [\ref{bib:Gra.Har.1}].
In the case when there is enough boundary data, one can reduce a $3D$ Korn inequality to three
$2D$ Korn inequalities, taking the body cross sections parallel to the coordinate planes, which is the same as
fixing one of the variables in the inequality. As shown in [\ref{bib:Gra.Har.1}]
this approach works for cylindrical shells,
and thus we can expect it to work for washers too, with a difference that the asyptotics of the Korn constant
in the second Korn inequality drops from $h^{1.5}$ to $h^2.$
The main difficulty in proving the sharp Korn inequalities for cylindrical shells
in [\ref{bib:Gra.Har.1}], was the lack of appropriate $2D$ Korn inequalities,
as all previously known inequalities for rectangles had a constant scaling like $h^2.$
In the case of a washer the main difficulty is that
upon passing from Cartesian coordinates to cylindrical ones, an extra factor $\rho$, the Jacobin
occurs in all integrals, which is not controllable from below when the inner radius $r$ of the washer
goes to zero. This is in contrast to cylindrical shells, as in that case $\rho$ stays close to the
radius of the inner circle and can be regarded as a constant.
We demonstrate below how one can overcome that singularity issue.
We first prove Theorem~\ref{th:2.1} and as will be seen later Theorem~\ref{th:2.2} is a direct consequence
of Theorem~\ref{th:2.1} and the Poincar\'e inequality.
\begin{proof}[Proof of Theorem~\ref{th:2.1}]
 As mentioned in the beginning of the section we consider the
$\rho,\theta,z\equiv const$ cross sections of the washer, which is the same as proving the inequality block-by block.
\begin{itemize}
\item{}[\textbf{The $z\equiv const$ cross section}]. The cross section $z\equiv const$ corresponds to the $2\times 2$ block of the gradient matrix built by the elements $11, 12, 21, 22.$ We aim to prove, that
\begin{equation}
\label{4.1}
\left\|\sqrt{\rho}\frac{u_{\rho,\theta}-u_\theta}{\rho}\right\|^2+\|\sqrt{\rho}u_{\theta,\rho}\|^2\leq 12\|\sqrt{\rho}e(\mathbf{u})\|^2.
\end{equation}
Denote $T=(r,R)\times (0,2\pi).$ Then for any fixed $z\in [0,h]$ we have integrating by parts and using the boundary conditions in
$\rho$ and the periodicity in $\theta$, that
\begin{align}
\label{4.2}
\int_{T}\rho\frac{u_{\rho,\theta}-u_\theta}{\rho}u_{\theta,\rho}d\rho d\theta
&=\int_{T}(u_{\rho,\theta}-u_\theta)u_{\theta,\rho}d\rho d\theta\\ \nonumber
&=\int_{T}u_{\rho,\theta}u_{\theta,\rho}d\rho d\theta\\ \nonumber
&=-\int_{T}u_{\rho}u_{\theta,\rho\theta}d\rho d\theta\\ \nonumber
&=\int_{T}u_{\rho,\rho}u_{\theta,\theta}d\rho d\theta\\ \nonumber
&=\int_{T}\rho u_{\rho,\rho}\frac{u_{\theta,\theta}+u_\rho}{\rho}d\rho d\theta-\int_{T}u_{\rho,\rho}u_\rho d\rho d\theta,\\ \nonumber
\end{align}
thus integrating the last inequality in $z\in[0,h]$ and applying the Schwartz inequality to both summands we discover
\begin{equation}
\label{4.3}
\left|\int_{\Omega}\rho\frac{u_{\rho,\theta}-u_\theta}{\rho}u_{\theta,\rho}d\rho d\theta\right|\leq
\|\sqrt{\rho}e(\mathbf{u})\|^2+\|\sqrt{\rho}e(\mathbf{u})\|\cdot\left\|\frac{u_\rho}{\sqrt{\rho}}\right\|.
\end{equation}
We then get utilizing (\ref{4.3}),
\begin{align}
\label{4.4}
\left\|\sqrt{\rho}\frac{u_{\rho,\theta}-u_\theta}{\rho}\right\|^2+\|\sqrt{\rho}u_{\theta,\rho}\|^2&=
\int_{\Omega}\rho\left(\frac{u_{\rho,\theta}-u_\theta}{\rho}+u_{\theta,\rho}\right)^2-
2\int_{\Omega}\rho\frac{u_{\rho,\theta}-u_\theta}{\rho}u_{\theta,\rho}\\ \nonumber
&\leq 4\|\sqrt{\rho}e(\mathbf{u})\|^2+2\|\sqrt{\rho}e(\mathbf{u})\|^2+
2\|\sqrt{\rho}e(\mathbf{u})\|\cdot\left\|\frac{u_\rho}{\sqrt{\rho}}\right\|\\ \nonumber
&=6\|\sqrt{\rho}e(\mathbf{u})\|^2+2\|\sqrt{\rho}e(\mathbf{u})\|\cdot\left\|\frac{u_\rho}{\sqrt{\rho}}\right\|.
\end{align}
It remains to bound $\left\|\frac{u_\rho}{\sqrt{\rho}}\right\|$ in terms of $\|\sqrt{\rho}e(\mathbf{u})\|$ and for that purpose let us prove, that
\begin{equation}
\label{4.4.5}
\left\|\frac{u_\rho}{\sqrt{\rho}}\right\|\leq 3\|\sqrt{\rho}e(\mathbf{u})\|.
\end{equation}
By the periodicity in $\theta,$ we can write the displacement $\mathbf{u}=(u_\rho,u_\theta,u_z)$ in Fourier space in $\theta$ in $W^{1,2}(\Omega).$
It is clear, that if $\mathbf{U}=\sum_{n=0}^\infty (\mathbf{U}_1^n(\rho,z)\cos(n\theta)+\mathbf{U}_2^n(\rho,z)\sin(n\theta)),$
then all inequalities under consideration
separate in the variable $n\in\{0,1,\dots\}$ and thus inequality (\ref{4.4.5}) is sufficient to prove for the case
$\mathbf{U}=\mathbf{U}_1(\rho,z)\cos(n\theta)+\mathbf{U}_2(\rho,z)\sin(n\theta),$ for some fixed $n\in\{0,1,\dots\}.$

Consider 2 cases.\\
\textbf{Case 1.} $n=0.$ In this case there is no $\theta$ dependence, and thus the term $\frac{u_\rho}{\rho}$ is the $22$
element of the symmetrized gradient $e(u),$ which implies
\begin{equation}
\label{4.5}
\left\|\frac{u_\rho}{\sqrt{\rho}}\right\|\leq \|\sqrt{\rho}e(\mathbf{u})\|.
\end{equation}
\textbf{Case 2.} $n\geq1.$ Denote
$$u_\rho=a_\rho(\rho,z)\cos(n\theta)+b_\rho(\rho,z)\sin(n\theta),\qquad u_\theta=a_\theta(\rho,z)\cos(n\theta)+b_\theta(\rho,z)\sin(n\theta).$$
We have on one hand by inequality (\ref{4.4}), that
\begin{equation*}
\left\|\frac{u_{\rho,\theta}-u_\theta}{\sqrt{\rho}}\right\|^2=\pi\int_{r}^R\int_{0}^h
\frac{|nb_{\rho}+a_\theta|^2+|na_\rho+b_\theta|^2}{\rho}dzd\rho\leq 6\|\sqrt{\rho}e(\mathbf{u})\|^2
+2\|\sqrt{\rho}e(\mathbf{u})\|\cdot\left\|\frac{u_\rho}{\sqrt{\rho}}\right\|
\end{equation*}
and on the other hand
\begin{equation*}
\left\|\frac{u_{\theta,\theta}+u_\rho}{\sqrt{\rho}}\right\|^2=\pi\int_{r}^R\int_{0}^h
\frac{|b_{\rho}-na_\theta|^2+|a_\rho-nb_\theta|^2}{\rho}dzd\rho\leq \|\sqrt{\rho}e(\mathbf{u})\|^2,
\end{equation*}
thus multiplying the first inequality by $n^2$ and summing the two inequalities and applying the triangle inequality, we obtain
\begin{align*}
(n^2+1)^2\left\|\frac{u_\rho}{\sqrt{\rho}}\right\|^2&=\pi(n^2+1)^2\int_{r}^R\int_{0}^h\frac{|a_\rho|^2+|b_{\rho}|^2}{\rho}dzd\rho\\
&\leq\frac{\pi}{2}\int_{r}^R\int_{0}^h \frac{|a_\rho-nb_\theta|^2+|n^2a_\rho+nb_\theta|^2+|b_{\rho}-na_\theta|^2+|n^2b_{\rho}+na_\theta|^2}{\rho}dzd\rho\\
&\leq \left(6n^2+1\right)\|\sqrt{\rho}e(\mathbf{u})\|^2+2n^2\|\sqrt{\rho}e(\mathbf{u})\|\cdot\left\|\frac{u_\rho}{\sqrt{\rho}}\right\|,
\end{align*}
which yields
\begin{align}
\label{4.6}
\left\|\frac{u_\rho}{\sqrt{\rho}}\right\|^2&\leq\frac{6}{n^2+1}\|\sqrt{\rho}e(\mathbf{u})\|^2+
\frac{2}{n^2+1}\|\sqrt{\rho}e(\mathbf{u})\|\cdot\left\|\frac{u_\rho}{\sqrt{\rho}}\right\| \\ \nonumber
&\leq 3\|\sqrt{\rho}e(\mathbf{u})\|^2+\|\sqrt{\rho}e(\mathbf{u})\|\cdot\left\|\frac{u_\rho}{\sqrt{\rho}}\right\|,
\end{align}
which then yields the estimate (\ref{4.4.5}).
Therefore, combining Case 1 and Case 2 we obtain
$$
\left\|\frac{u_\rho}{\sqrt{\rho}}\right\|^2\leq 3\|\sqrt{\rho}e(u)\|^2.
$$
In conclusion, estimates (\ref{4.4}) and (\ref{4.4.5}) imply (\ref{4.1}).

\item{}[\textbf{The $\theta\equiv const$ cross section}]. The cross section $\theta\equiv const$ corresponds to the $2\times 2$ block of the gradient matrix built by the elements $11, 13, 31, 33.$ Let us prove in this case, that
\begin{equation}
\label{4.8}
\|\sqrt{\rho}u_{\rho,z}\|^2+\|\sqrt{\rho}u_{z,\rho}\|^2
\leq C\left(\frac{\|\sqrt{\rho}u_z\|\cdot\|\sqrt{\rho}e(\mathbf{u})\|}{h}+\|\sqrt{\rho}e(\mathbf{u})\|^2\right).
\end{equation}
Note, that inequality (\ref{4.8}) is exactly Theorem~\ref{th:2.1} written for the displacement $\mathbf{U}=(u_z,u_\rho)$ in the
variable $(z,\rho)$ in the rectangular domain $T=(0,h)\times(r,R)$ and then integrated in $\theta\in[0,2\pi].$

\item{}[\textbf{The $\rho\equiv const$ cross section}].
The cross section $\rho\equiv const$ corresponds to the $2\times 2$ block of the gradient matrix built
by the elements $22, 23, 32, 33.$ We aim to prove, that
\begin{equation}
\label{4.9}
\|\sqrt{\rho}u_{z,\theta}\|^2+\left\|\sqrt{\rho}\frac{u_{\theta,z}}{\rho}\right\|^2\leq C\left(\frac{\|\sqrt{\rho}u_z\|\cdot\|\sqrt{\rho}e(\mathbf{u})\|}{h}+\|\sqrt{\rho}e(\mathbf{u})\|^2\right).
\end{equation}
For a fixed $\rho>0$ consider the vector field
$$\mathbf{U}(z,\theta_1)=(f(\rho,\theta_1,z),g(\rho,\theta_1,z))=(u_z(\rho,\theta,z), u_\theta(\rho,\theta,z))$$
defined on the rectangle $(z,\theta_1)\in[0,h]\times[0,2\pi\rho]=T_1,$ where $\theta_1=\rho\theta$. We have that
\begin{equation}
\label{4.10}
\nabla_{z,\theta_1}\mathbf{U}(z,\theta_1)=
\begin{bmatrix}
f_{,z}(\rho,\theta_1,z) & f_{,\theta_1}(\rho,\theta_1,z)\\
g_{,z}(\rho,\theta_1,z) & g_{,\theta_1}(\rho,\theta_1,z)
\end{bmatrix}=
\begin{bmatrix}
u_{z,z}(\rho,\theta,z) & \frac{1}{\rho}u_{z,\theta}(\rho,\theta,z) \\
u_{\theta,z}(\rho,\theta,z) & \frac{1}{\rho}u_{\theta,\theta}(\rho,\theta,z)
\end{bmatrix}
\end{equation}
Note, that both spaces $V_1$ and $V_2$ are compatible with the requirements of Theorem~\ref{3.1}, thus we get
 \begin{equation}
 \label{4.11}
 \|\nabla_{z,\theta_1}\mathbf{U}\|_{L^2(T_1)}^2\leq
 100\left(\frac{\|f\|_{L^2(T_1)}\cdot\|e(\mathbf{U})\|_{L^2(T_1)}}{h}+\|e(\mathbf{U})\|_{L^2(T_1)}^2\right).
 \end{equation}
 Next, denoting $T=[0,h]\times[0,2\pi]$ we get by the triangle inequality
 \begin{align}
 \label{4.12}
 \|&e(\mathbf{U})\|_{L^2(T_1)}^2\\ \nonumber
 &=\int_{T_1}\left(u_{z,z}^2(\rho,\theta,z)+\left(\frac{1}{\rho}u_{\theta,\theta}(\rho,\theta,z)\right)^2+
 \frac{1}{4}\left(u_{\theta,z}(\rho,\theta,z)+\frac{1}{\rho}u_{z,\theta}(\rho,\theta,z)\right)^2\right)dzd\theta_1\\ \nonumber
 &=\rho\int_{T}\left(u_{z,z}^2(\rho,\theta,z)+\left(\frac{1}{\rho}u_{\theta,\theta}(\rho,\theta,z)\right)^2+
 \frac{1}{4}\left(u_{\theta,z}(\rho,\theta,z)+\frac{1}{\rho}u_{z,\theta}(\rho,\theta,z)\right)^2\right)dzd\theta\\ \nonumber
 &\leq\rho\int_{T}\left(u_{z,z}^2+2\left(\frac{1}{\rho}(u_{\theta,\theta}+u_\rho)\right)^2
 +\frac{2u_\rho^2}{\rho^2}+\frac{1}{4}\left(u_{\theta,z}+\frac{1}{\rho}u_{z,\theta}\right)^2\right)dzd\theta\\ \nonumber
 &=\rho S^2(\rho).
  \end{align}
Thus we get from (\ref{4.10})-(\ref{4.12}) the estimate
$$
\int_{T_1}\left[u_{\theta,z}^2(\rho,\theta,z)+\frac{u_{z,\theta}^2}{\rho^2}\right]dzd\theta_1\leq
100\left(\frac{S(\rho)}{h}\left(\rho\int_{T_1}u_z^2(\rho,\theta,z)dzd\theta_1\right)^{\frac{1}{2}}+\rho S^2(\rho)\right),
$$
from where we get
\begin{equation}
\label{4.13}
\int_{T}\left[u_{\theta,z}^2(\rho,\theta,z)+\frac{u_{z,\theta}^2}{\rho^2}\right]dzd\theta\leq 100\left(\frac{S(\rho)}{h}\left(\int_{T}u_z^2(\rho,\theta,z)dzd\theta\right)^{\frac{1}{2}}+S^2(\rho)\right).
\end{equation}
Multiplying the last inequality by $\rho$ and integrating in $\rho$ from $r$ to $R$ and utilizing the Schwartz inequality for the product
$S(\rho)\left(\int_{T}u_z^2(\rho,\theta,z)dzd\theta\right)^{\frac{1}{2}}$ we get
\begin{equation}
\label{4.14}
\|\sqrt{\rho}u_{\theta,z}\|^2+\left\|\frac{u_{z,\theta}}{\sqrt{\rho}}\right\|^2\leq 100\left(\frac{\|\sqrt{\rho}u_z\|}{h}\left(\int_{r}^R\rho S^2(\rho)d\rho\right)^{\frac{1}{2}}+\int_{r}^R\rho S^2(\rho)d\rho\right).
\end{equation}
It is clear from (\ref{4.12}), that
\begin{equation}
\label{4.15}
\int_{r}^R\rho S^2(\rho)d\rho\leq 2\|\sqrt{\rho}e(\mathbf{u})\|^2+2\left\|\frac{u_\rho}{\sqrt{\rho}}\right\|^2.
\end{equation}
Inequalities (\ref{4.15}) and (\ref{4.4.5}) imply
$$\int_{r}^R\rho S^2(\rho)d\rho\leq 8\|\sqrt{\rho}e(\mathbf{u})\|^2,$$
hance (\ref{4.9}) is derived from (\ref{4.14}),
which completes the proof of the last case.
In conclusion, a combination of estimates (\ref{4.1}), (\ref{4.8}) and (\ref{4.9}) completes the proof of Theorem~\ref{th:2.1}.
\end{itemize}
\end{proof}

\begin{proof}[Proof of Theorem~\ref{th:2.2}]
We follow different approaches for the subspaces $V_1$ and $V_2,$ thus consider
the cases $\mathbf{u}\in V_1$ and $\mathbf{u}\in V_2$ separately.\\

\textbf{Case $\mathbf{u}\in V_1.$}
Denote $\mathbf{u}=(v_x,v_y,v_z)$ in Cartesian coordinates. Then it is clear, that $v_z=u_z.$
Observe, that a translation of the $v_z$ component of the displacement $\mathbf{u}$ by a constant
does not affect inequality (\ref{2.3}), thus because there are no boundary conditions imposed on the $u_z$ component of the displacement,
we can translate the component $v_z$ by a suitable constant $\lambda$ to make it satisfy the zero average condition
\begin{equation}
\label{5.16}
\int_{\Omega}v_z(x,y,z)dxdydz=0.
\end{equation}
Due to the zero average condition (\ref{5.16}) we get by the Poincer\'e inequality, that
\begin{equation}
\label{5.17}
\int_{\Omega}|v_z(x,y,z)|^2dxdydz\leq C(R) \int_{\Omega}|\nabla_{x,y,z} v_z(x,y,z)|^2dxdydz,
\end{equation}
in Cartesian coordinates. Next, by the identity
$$|\nabla_{x,y,z} v_z(x,y,z)|^2=\left|\frac{\partial u_z}{\partial \rho}\right|^2+
\frac{1}{\rho^2}\left|\frac{\partial u_z}{\partial \theta}\right|^2+\left|\frac{\partial u_z}{\partial z}\right|^2,$$
we get the equality
\begin{equation}
\label{5.18}
\int_{\Omega}|\nabla_{x,y,z} v_z(x,y,z)|^2dxdydz=\int_{\Omega}\rho|\nabla_{\rho,\theta,z} u_z(\rho,\theta,z)|^2d\rho d\theta dz.
\end{equation}
A combination of (\ref{5.17}) and (\ref{5.18}) gives
\begin{equation}
\label{5.19}
\|\sqrt{\rho}u_z\|^2\leq C(R)\|\sqrt{\rho}\nabla u_z|^2
\end{equation}
in cylindrical coordinates. Inequality (\ref{2.3}) is now a direct consequence of (\ref{2.2}) and (\ref{5.19}).\\

\textbf{Case $\mathbf{u}\in V_2.$} We again consider two subcases here.\\
\textbf{Subcase $2r\geq R. $} This is the easier case and is done in terms of Poincar\'e inequality. Denote $T=[0,2\pi]\times[0,h].$ Due to the
boundary condition $u_z(r,\theta,z)=0,$ we have for any fixed $(\rho,\theta,z)\in\Omega,$ that
$$u_z(\rho,\theta,z)=\int_r^\rho u_{z,t}(t,\theta,z)dt,$$
thus we get by the Schwartz inequality
\begin{align*}
u_z^2(\rho,\theta,z)&=\left(\int_r^\rho u_{z,t}(t,\theta,z)dt\right)^2\\
&\leq \int_{r}^\rho\frac{1}{t}dt\int_r^\rho tu_{z,t}^2(t,\theta,z)dt\\
&\leq \ln{\frac{R}{r}}\int_r^R tu_{z,t}^2(t,\theta,z)dt\\
&\leq \ln2\int_r^R tu_{z,t}^2(t,\theta,z)dt.
\end{align*}
Multiplying the last inequality by $\rho$ and integrating in $\rho$ from $r$ to $R$ we obtain

\begin{align}
\label{5.21}
\int_r^R \rho u_z^2(\rho,\theta,z)d\rho &\leq \frac{\ln2}{2}(R^2-r^2)\int_r^R tu_{z,t}^2(t,\theta,z)dt\\ \nonumber
&\leq R^2\int_r^R tu_{z,t}^2(t,\theta,z)dt.
\end{align}
It remains to integrate inequality (\ref{5.21}) in the variables $(\theta,z)$ over the rectangle $T$ to get
$$\|\sqrt{\rho}u_z\|^2\leq R^2\|\sqrt{\rho}u_{z,\rho}\|^2\leq R^2\|\sqrt{\rho}\nabla\mathbf{u}|^2,$$
as the element $u_{z,\rho}$ is the 31 element of the gradient matrix $\nabla\mathbf{u}.$ The last inequality together with (\ref{2.2})
imply (\ref{2.3}).\\
\textbf{Subcase $2r<R.$} This is the trickier case and is done by the Hardy inequality proven in Lemma~\ref{lem:3.5}. As $R<2\cdot\frac{R+r}{2},$
then by the analogy of the previous subcase if we choose the outer thin boundary points of the washer as the starting points in the integration in
the Poincar\'e inequality, then we get an estimate like (\ref{5.21}) over the interval $\left[\frac{R+r}{2},R\right],$ namely we get for any
$(\theta,z)\in T,$ the following estimate:
\begin{equation}
\label{5.22}
\int_{\frac{R+r}{2}}^R \rho u_z^2(\rho,\theta,z)d\rho\leq  R^2\int_{\frac{R+r}{2}}^R tu_{z,t}^2(t,\theta,z)dt.
\end{equation}
On the other hand applying Lemma~\ref{lem:3.5} to the function $u_z(t,\theta,z)$ in the interval $[r,R]$ we obtain

\begin{equation}
\label{5.23}
\int_r^{\frac{R+r}{2}}\rho u_z^2(\rho,\theta,z)d\rho \leq  4\int_{\frac{R+r}{2}}^R \rho u_z^2(\rho,\theta,z)d\rho+R^2\int_r^R tu_{z,t}^2(t,\theta,z)dt,
\end{equation}
thus combining (\ref{5.22}) and (\ref{5.23}) we get
\begin{equation}
\label{5.24}
\int_r^R\rho u_z^2(\rho,\theta,z)d\rho \leq  5R^2\int_r^R \rho u_{z,\rho}^2(\rho,\theta,z)d\rho.
\end{equation}
In integration of the last inequality in the $(\theta,z)$ variables over the rectangle $T$ yields
\begin{equation}
\label{5.25}
\|\sqrt{\rho}u_z\|^2\leq 5R^2\|\sqrt{\rho}u_{z,\rho}\|^2.
\end{equation}
Finally, the estimate (\ref{2.3}) follows from (\ref{5.25}) and (\ref{2.2}). The proof of the theorem is finished now.

\end{proof}

\begin{proof}[Proof of Theorem~\ref{th:2.3}]
Let us assume that there is no $\theta$ dependence in the the Ansatz, thus we seek it in the form
$$u=(u_\rho,u_\theta,u_z)=(f(\rho,z),0,g(\rho,z)).$$ Plugging this vector field into formula (\ref{2.1}) we obtain
\begin{equation}
\label{4.22}
\nabla u=
\begin{bmatrix}
f_{,\rho} & 0 & f_{,z}\\
0 & \frac{f}{\rho} & 0\\
g_{,\rho} & 0 & g_{,z}
\end{bmatrix}.
\end{equation}
Then following the idea of Grabovsky and Truskinovsky in [\ref{bib:Gra.Tru.}], which is expanding
the Ansatz in the powers of $h$ and $z$ and keeping only up to the first order terms we arrive at the
Kirchiff Ansatz
\begin{equation}
\label{4.23}
f(\rho,z)=-z\varphi'(\rho),\qquad g(\rho,z)=\varphi(\rho),
\end{equation}
where $\varphi$ is a compactly supported smooth function on $[\frac{R+r}{2},R]$ such that
$\int_r^R |\varphi'(\rho)|^2d\rho$ and $\int_r^R|\varphi''(\rho)|^2d\rho$ are both of order one.
Then it is clear that
\begin{equation}
\label{4.24}
\nabla u=
\begin{bmatrix}
-z\varphi''(\rho) & 0 & -\varphi'(\rho)\\
0 & \frac{-z\varphi'(\rho)}{\rho} & 0\\
\varphi'(\rho) & 0 & 0
\end{bmatrix},
\end{equation}
thus we obtain $\|\sqrt{\rho}e(u)\|^2\sim h^3$ and $\|\sqrt{\rho}\nabla u\|^2\sim h,$ therefore inequality (\ref{2.3}) is sharp in terms of the power of $h.$
For inequality (\ref{2.2}) we choose the Ansatz $u=(f,0,g),$ [\ref{bib:Harutyunyan}, Remark~1.5], where
\begin{equation}
\label{4.25}
f(\rho,z)=-\frac{z}{h^\alpha}\varphi'\left(\frac{\rho}{h^\alpha}\right),\qquad g(\rho,z)=\varphi\left(\frac{\rho}{h^\alpha}\right),\quad\alpha\in\left[0,\frac{1}{2}\right],
\end{equation}
and again $\varphi$ is a compactly supported smooth function on $[\frac{R+r}{2},R]$ such that
$\int_r^R |\varphi'(\rho)|^2d\rho$ and $\int_r^R|\varphi''(\rho)|^2d\rho$ are both of order one.
It is clear, that
\begin{equation}
\label{4.26}
\nabla u=
\begin{bmatrix}
-\frac{z}{h^{2\alpha}}\varphi''(\rho) & 0 & -\frac{1}{h^{\alpha}}\varphi'(\rho)\\
0 & \frac{z}{h^{\alpha}\rho}\varphi'(\rho) & 0\\
\frac{1}{h^{\alpha}}\varphi'(\rho) & 0 & 0
\end{bmatrix},
\end{equation}
hence we get $\|\sqrt{\rho}e(u)\|\sim h^{1.5-2\alpha}$, $\|\sqrt{\rho}\nabla u\|^2\sim h^{1-2\alpha}$ and $\|\sqrt{\rho}g\|\sim h^{0.5}.$
Therefore inequality (\ref{2.2}) is sharp in terms of the power of $h.$ Theorem~\ref{th:2.3} is proved now.
\end{proof}

\section*{Acknowledgements.}
The author is grateful to Graeme Milton and the University of Utah for support.

\end{document}